\documentclass[12pt,a4paper,reqno,twoside]{amsart}

\usepackage[english]{babel}
\usepackage{stmaryrd}
\usepackage{dsfont}
\usepackage[symbol*,ragged]{footmisc}
\usepackage[colorlinks,linkcolor=red,anchorcolor=blue,citecolor=blue,urlcolor=blue]{hyperref}
\usepackage{color,xcolor}

\usepackage{geometry}
\usepackage{amssymb}
\usepackage{amsmath}
\usepackage{mathrsfs}
\usepackage{amsfonts}
\usepackage{epsfig}

\usepackage{amsthm}
\usepackage{amsxtra}
\usepackage{bbding}
\usepackage{epsfig}
\usepackage{graphicx}
\usepackage{latexsym}
\usepackage{mathbbol}
\usepackage{bbold}

\usepackage{pifont}
\usepackage{wasysym}
\usepackage{skull}
\usepackage{float}

\DeclareSymbolFontAlphabet{\mathbb}{AMSb}
\DeclareSymbolFontAlphabet{\mathbbol}{bbold}

\usepackage{amscd}
\usepackage[all]{xy}
\allowdisplaybreaks[4]
\usepackage{setspace}

\theoremstyle{plain}
\newtheorem{theorem}{Theorem}[section]
\newtheorem{proposition}{Proposition}[section]
\newtheorem{lemma}[proposition]{Lemma}
\newtheorem{corollary}[theorem]{Corollary}
\newtheorem*{corollary*}{Corollary}

\theoremstyle{remark}
\newtheorem*{remark*}{\normalfont\scshape Remark}
\newtheorem*{notation}{\normalfont\scshape Notation}

\numberwithin{equation}{section}
\addtocounter{footnote}{1}

\renewcommand{\footnoterule}{
  \kern -3pt
  \hrule width 2.5in height 0.4pt
  \kern 3pt
}

\newcommand{\E}{\mathbb{E}}
\newcommand{\N}{\mathbb{N}}
\newcommand{\Z}{\mathbb{Z}}

\makeatletter
\@ifundefined{MakeUppercase}{}{}
\makeatother

\begin{document}
	
\title[A density version of quaternary Goldbach problem]
	  {A density version of quaternary Goldbach problem}

\author[Xiaoyang Hu, Meng Gao]
       {Xiaoyang Hu \quad \& \quad Meng Gao}

\address{[Xiaoyang Hu] School of Mathematics and Statistics, Jiangsu Normal University,
Xuzhou 221116, People's Republic of China}

\email{\textcolor{blue}{2220502116@cnu.edu.cn}}

\address{[Meng Gao] (Corresponding author) Department of Mathematics, China University of Mining and Technology,
Beijing 100083, People's Republic of China}

\email{\textcolor{blue}{meng.gao.math@gmail.com}}


\footnotetext[1]{Meng Gao is the corresponding author.  

\quad\,\,
{\textbf{MR(2020) Subject Classification}}:11P32, 11B30.

}

\begin{abstract}
Let $\mathcal{P}$ denote the set of all primes, and let $\underline\delta(P)$ denote the relative lower density of a subset $P$ in $\mathcal{P}$. Suppose that $P_1, P_2, P_3, P_4$ are four subsets of primes with $\underline\delta(P_1)+\underline\delta(P_2)>1$ and $ \underline\delta(P_3)+\underline\delta(P_4)>1.$ Then for every sufficiently large even integer $n$, there exist primes $p_i \in P_i$ $(i=1,2,3,4)$ such that $n=p_1 +p_2 +p_3+p_4$. The condition is the best possible.
\end{abstract}

\maketitle

\section{Introduction}

The famous binary Goldbach conjecture states that every even integer greater than $2$ can be written as the sum of two primes. The ternary version concerns the representation of odd integers as sums of three primes. In 1937, Vinogradov \cite{VG} proved that the ternary Goldbach conjecture holds for all sufficiently large odd integers without assuming the Generalized Riemann Hypothesis (GRH). In 2013, Helfgott \cite{HG} gave a complete proof of the ternary Goldbach conjecture for all odd integers greater than 5.

On the other hand, after the pioneering work on Roth’s theorem in the primes by Green \cite{Gre} and the celebrated Green–Tao theorem \cite{GT}, additive properties of positive density subsets of the primes have been explored in many studies, such as density versions of the ternary Goldbach problem \cite{LP,Shao,Shen}, density versions of the binary Goldbach problem \cite{AS}, density versions of the Waring–Goldbach problem \cite{Gao}, and especially the quadratic Waring–Goldbach problem \cite{LZL,Tan,ZGH}, as well as some sparse subsets of primes \cite{Gri,MS,MMS,SDP,Te}.

The purpose of this paper is to prove a density version of quaternary Goldbach problem. Let $ \mathcal{P} $ be the set of all primes, and define the relative lower density of $P$ in $\mathcal{P}$ by $$ \underline{\delta}(P):=\liminf \limits_{N\rightarrow\infty}\dfrac{|P\cap[N]|}{|\mathcal{P}\cap[N]|}, $$ where $[N]:=\{1,\ldots,N\}$. Recently, Lacey, Mousavi, Rahimi and Vempati \cite{LMRV} proved the following density version of quaternary Goldbach problem. 

\begin{theorem}[{\cite[Theorem 1.2]{LMRV}}]\label{LRMV1}
Let $P$ be a subset of primes with density $\underline\delta(P)>\frac{1}{2}$. Then for every sufficiently large even integer $n$, there exist $p_i \in P$ $(i=1,2,3,4)$ such that $n = p_1 + p_2 + p_3 + p_4$.
\end{theorem}

Motivated by their result, our main results are as follows:

\begin{theorem}\label{Theorem-1}
Let $P_1, P_2, P_3, P_4$ be four subsets of primes with
\begin{equation*}
\underline\delta(P_1)+\underline\delta(P_2)>1,  \ \underline\delta(P_3)+\underline\delta(P_4)>1.
\end{equation*}
Then for every sufficiently large even integer $n$, there exist $p_i \in P_i$ $(i=1,2,3,4)$ such that $n = p_1 + p_2 + p_3 + p_4$.
\end{theorem}

Note that Theorem \ref{Theorem-1} immediately implies Theorem \ref{LRMV1}. 
We remark that the constant 1 in Theorem \ref{Theorem-1} is sharp in two different ways, 
and the statement fails if we replace the symbol $>$ by $\geq$. 
In fact, we either take
$$P_1 = P_2 = P_3 = \{ p \in \mathcal{P} : p \equiv 1 \pmod 3 \}, \quad P_4 = \mathcal{P}\setminus \{3\},$$
or take
$$P_1 = \varnothing, \quad P_2 = P_3 = P_4 = \mathcal{P}\setminus \{2\}.$$
Then $\underline{\delta}(P_1) + \underline{\delta}(P_2) = 1$ and $\underline{\delta}(P_3) + \underline{\delta}(P_4) > 1$. 
But any even integer $n \equiv 0 \pmod 3$ cannot be written as $p_1 + p_2 + p_3 + p_4$ with $p_i \in P_i$ $(i = 1, 2, 3, 4)$.

In particular, we have the following.

\begin{corollary}
 Let $P_1, P_2, P_3, P_4$ be four subsets of primes with
\begin{equation*}
\underline\delta(P_1)\geq\frac{1}{2}, \ \underline\delta(P_2)>\frac{1}{2},  \ \underline\delta(P_3)\geq \frac{1}{2}, \ \underline\delta(P_4)>\frac{1}{2}.
\end{equation*}
Then for every sufficiently large even integer $n$, there exists $p_i\in P_ i$ $(i=1,2,3,4)$ such that $n=p_1 +p_2 +p_3+p_4$.   
\end{corollary}

Theorem \ref{Theorem-1} is proved using Green's transference principle from additive combinatorics developed in \cite{Gre}. The key to applying the transference principle is to solve the local problem associated with Theorem \ref{Theorem-1}. Based on the ideas of \cite[Section 2]{LP}, we prove a four-function variant of \cite[Theorem 1.2]{LP} that applies to the local version of the quaternary Goldbach problem (Corollary \ref{coro1}). The precise statement is as follows.

\begin{theorem}\label{Theorem-2}
Let $q$ be a positive odd integer with $(q, 6) = 1$.  Let $f_1, f_2, f_3,f_4$ be four real-valued functions on $\Z_q^*$. Then for any $n \in \Z_q$, there exist $x_1, x_2, x_3,x_4 \in \Z_q^*$ such that $n = x_1+ x_2+ x_3+x_4 $ and
\begin{equation*}
f_1(x_1)+f_2(x_2)\geq\E_{x\in\Z_q^*}(f_1(x)+f_2(x)),  
\end{equation*}
\begin{equation*}
f_3(x_3)+f_4(x_4)\geq\E_{x\in\Z_q^*}(f_3(x)+f_4(x)).
\end{equation*}
\end{theorem}

We will give the proof of Theorem \ref{Theorem-2} and the local results we need in Section 2, and prove Theorem \ref{Theorem-1} in Section 3.

\begin{notation}
	For a positive integer $q$, let $\mathbb{Z}_{q}=\mathbb{Z}/q\mathbb{Z}$ and $\mathbb{Z}_{q}^{\ast}=\{ a\in \mathbb{Z}_{q}:(a,q)=1 \}$.
	For a set $\mathcal{A}$, we write $\mathbf{1}_{\mathcal{A}}(x)$ for its characteristic function. The letter $p$, with or without subscript, denotes a prime number. As usual, we use $\varphi(n)$ to denote Euler's function. Let $f:\mathcal{B}\rightarrow\mathbb{C}$ be a function and $\mathcal{B}_1$ be a non-empty finite subset of $\mathcal{B}$. Denote by $\mathbb{E}_{x\in\mathcal{B}_1}f(x)$ the average value of $f$ on $\mathcal{B}_1$, i.e.,
	\begin{equation*}
		\mathbb{E}_{x\in\mathcal{B}_1}f(x)=\dfrac{1}{|\mathcal{B}_1|}\sum\limits_{x\in\mathcal{B}_1}f(x).
	\end{equation*}
	For a function $f:\mathbb{Z}_N\rightarrow \mathbb{C}$, define the Fourier transform on $\mathbb{Z}_N$ by
	\begin{equation*}
		\tilde{f}(r) = \sum_{x \in \mathbb{Z}_N} f(x) e(-xr/N),
	\end{equation*}
	where $e(x) = e^{2\pi\sqrt{-1}x}$. Also, for functions $f,g$ over $\mathbb{Z}_N$, we define convolution $f\ast g$ by
	\begin{equation*}
		f\ast g(x) = \sum_{y \in \mathbb{Z}_N} f(y) g(x-y).
	\end{equation*}
	It is easy to check that  $\widetilde{f\ast g}= \tilde{f}\cdot\tilde{g}$.
\end{notation}

\section{Local results}
\begin{proof}[Proof of Theorem \ref{Theorem-2}]
Let $K_1=\E_{x\in\Z_q^*}(f_1(x)+f_2(x))$, $K_2=\E_{x\in\Z_q^*}(f_3(x)+f_4(x))$.
    We use induction on the number of prime divisors of $q$.
    
    First, consider $q=p$ where $p\geq 5$ is a prime. Define $S_i=\sum_{x\in\Z_p^*}f_i(x)$ for $i=1,2,3,4$, then
\begin{equation*}
    K_1(p-1)=S_1+S_2, \ K_2(p-1)=S_3+S_4.
\end{equation*}

Assume on the contrary that there exists a counterexample $n \in \Z _p$ such that for any $x_1, x_2, x_3,x_4 \in \Z^*_ p$ with $x_1+ x_2+ x_3+x_4=n$, 
\begin{equation*}
    f_1(x_1)+f_2(x_2)<K_1 \ \text{or} \ f_3(x_3)+f_4(x_4)<K_2 .
\end{equation*}

In other words, all $y\in\Z_p$ satisfy at least one of the following statements.

\textbf{(a)} For all $x\in \Z^ *_ p\setminus\{y\}$, $ f_1(x)+f_2(y-x)< K_1 .$

\textbf{(b)} For all $x\in \Z^ *_ p\setminus\{n-y\}$, $f_3(x)+f_4(n-y-x)< K_2.   $

\textbf{Case 1.}
If statement \textbf{(a)} holds for $y=0$, then for any $x\in \Z^ *_ p$  we have  $ f_1(x)+f_2(-x)< K_1 $. Summing this inequality over all $x \neq 0$,  we obtain a contradiction:
\begin{equation*}
S_1+S_2=\sum_{x\in\Z_p^*}(f_1(x)+f_2(x))=\sum_{x\in\Z_p^*}(f_1(x)+f_2(-x))<K_1(p-1).
\end{equation*}

We remark that the following \textbf{Case 1'} is the symmetric version of Case 1, but for the convenience of the proof below, we still state it explicitly.

\textbf{Case 1'.}
If statement \textbf{(b)} holds for $y=n$, then for any $x\in \Z^ *_ p$  we have  $ f_3(x)+f_4(-x)< K_2 $. Summing this inequality over all $x \neq 0$,  we obtain a contradiction:
\begin{equation*}
S_3+S_4=\sum_{x\in\Z_p^*}(f_3(x)+f_4(x))=\sum_{x\in\Z_p^*}(f_3(x)+f_4(-x))<K_2(p-1).
\end{equation*}

\textbf{Case 2.}
If statement \textbf{(b)} holds for $y=0$ and $n=0$, then for any $x\in \Z^ *_ p$  we have  $ f_3(x)+f_4(-x)< K_2 $. Summing this inequality over all $x \neq 0$,  we obtain a contradiction:
\begin{equation*}
S_3+S_4=\sum_{x\in\Z_p^*}(f_3(x)+f_4(x))=\sum_{x\in\Z_p^*}(f_3(x)+f_4(-x))<K_2(p-1).
\end{equation*}

By \textbf{Case 1} and \textbf{Case 2}, we know that $n\neq 0$ and the statement \textbf{(b)} holds for $y=0$. Therefore, for any $x\in \Z^ *_ p\setminus\{n\}$  we have  $ f_3(x)+f_4(n-x)< K_2 $. Summing this inequality over all $x \neq 0,n$,  we have
\begin{equation*}
	S_3+S_4-f_3(n)-f_4(n)=\sum_{x\in\Z_p^*\setminus\{n\}}(f_3(x)+f_4(n-x))<K_2(p-2).
\end{equation*}

Thus 
\begin{equation}\label{En}
K_2<f_3(n)+f_4(n).
\end{equation}

Since the statement \textbf{(a)} fails to hold for $y=0$, there exists $x_0\in \Z^*_p$  such that   $f_1(x_0)+f_2(-x_0)\geq K_1$. Thus $x_0-x_0+n+n= 2n $ is not a counterexample. 

Inequality (\ref{En}) means that the statement \textbf{(b)} fails to hold for $y=-n$. Therefore by the assumption of contradiction we have that the statement \textbf{(a)} holds for $y=-n$, i.e., for any $x\in \Z^*_ p\setminus\{-n\}$, we have $f_1(x)+f_2(-n-x)< K_1$. Similarly we have
\begin{equation}\label{E-n}
K_1<f_1(-n)+f_2(-n).
\end{equation}

Inequality (\ref{E-n}) means that the statement \textbf{(a)} fails to hold for $y=-2n$. Therefore by the assumption of contradiction we have that the statement \textbf{(b)} holds for $y=-2n$, i.e., for any $x\in \Z^*_ p\setminus\{3n\}$, we have $f_3(x)+f_4(3n-x)< K_2$. By \textbf{Case 1'}, we only need to consider $3n\not= 0 $. Similarly we have
\begin{equation}\label{E3n}
K_2<f_3(3n)+f_4(3n).
\end{equation}

Inequality (\ref{E3n}) means that the statement \textbf{(b)} fails to hold for $y=-5n$. Therefore by the assumption of contradiction we have that the statement \textbf{(a)} holds for $y=-5n$, i.e., for any $x\in \Z^ *_ p\setminus\{-5n\}$, we have $f_1(x)+f_2(-5n-x)< K_1$. By \textbf{Case 1}, we only need to consider $-5n\not= 0 $. Similarly we have
\begin{equation}\label{E-5n}
K_1<f_1(-5n)+f_2(-5n).
\end{equation}

Inequality (\ref{E-5n}) means that the statement \textbf{(a)} fails to hold for $y=-10n$. Therefore by the assumption of contradiction we have that the statement \textbf{(b)} holds for $y=-10n$, i.e., for any $x\in \Z^ *_ p\setminus\{11n\}$, we have $f_3(x)+f_4(11n-x)< K_2$. By \textbf{Case 1'}, we only need to consider $11n\not= 0 $. Similarly we have
\begin{equation}\label{E11n}
K_2<f_3(11n)+f_4(11n).
\end{equation}

Write $t_0(x)=x$, $t_1(x)=1-2x$ and $t_{k}(x)=t_1( t_{k-1}(x))$ when $k>1$. Observe the coefficient of 
$n$ in each of the inequalities (\ref{En}), (\ref{E-n}), (\ref{E3n}), (\ref{E-5n}) and (\ref{E11n}). We now proceed by induction on $k$ to show that either we obtain a contradiction by \textbf{Case 1} or \textbf{Case 1'}, or for all $k\in \N$ we have
\begin{equation}\label{E2k+1n}
    K_1<f_1(t_{2k+1}(1)n)+f_2(t_{2k+1}(1)n),\ t_{2k+1}(1)n\neq 0,
\end{equation}
\begin{equation}\label{E2kn}
    K_2<f_3(t_{2k}(1)n)+f_4(t_{2k}(1)n) ,\ t_{2k}(1)n\neq 0.
\end{equation}

For $k=0$, inequalities (\ref{E2k+1n}) and (\ref{E2kn}) follow from inequalities (\ref{En}) and (\ref{E-n}).

We now assume that inequalities (\ref{E2k+1n}) and (\ref{E2kn}) holds for $k-1$.

Again, induction hypothesis inequality (\ref{E2k+1n}) for $k-1$ means that the statement \textbf{(a)} fails to hold for $y=2t_{2k-1}(1)n$. Therefore by the assumption of contradiction we have that the statement \textbf{(b)} holds for $y=2t_{2k-1}(1)n$. Here $n-y=t_{2k}(1)n$, thus for any $x\in \Z^*_p\setminus\{t_{2k}(1)n\}$, we have $f_3(x)+f_4(t_{2k}(1)n-x)< K_2$. By \textbf{Case 1'}, we only need to consider $t_{2k}(1)n\not= 0 $. Similarly we have
\begin{equation}\label{E2kn2}
K_2<f_3(t_{2k}(1)n)+f_4(t_{2k}(1)n).
\end{equation}

Inequality (\ref{E2kn2}) means that the statement \textbf{(b)} fails to hold for $y=n-2t_{2k}(1)n$. By the assumption of contradiction, we have that the statement \textbf{(a)} holds for $y=n-2t_{2k}(1)n$. Here $y=t_{2k+1}(1)n$, thus for any $x\in \Z^*_ p\setminus\{t_{2k+1}(1)n\}$, we have $f_1(x)+f_2(t_{2k+1}(1)n-x)< K_1$. By \textbf{Case 1}, we only need to consider $t_{2k+1}(1)n\not= 0 $. Similarly we have
\begin{equation*}\label{E2k+1n2}
K_1<f_1(t_{2k+1}(1)n)+f_2(t_{2k+1}(1)n).
\end{equation*}
This completes the induction step on $k$.

Notice that $t_{k}(1)=(1- (-2)^{k+1})/3$ is the explicit formula for $t_{k}(1)$. Also, for $r,s\geq 0$, we have 
\begin{equation*}
2\cdot\frac{1- (-2)^{2r}}{3}+ 2\cdot\frac{1- (-2)^{2s+1}}{3}\equiv 1\pmod p  
\end{equation*}
is equivalent to
\begin{equation*}
2\cdot4^{r} -4^{s+1}\equiv 1\pmod p.
\end{equation*}

Since $p\geq 5$, the cyclic group $\langle  4\rangle$ is a subgroup of $\mathbb F_p^\times$. Therefore there exists $r>0$ such that
\begin{equation*}
4^{r} \equiv 1\pmod p.    
\end{equation*}
Let $s=r-1$. Then $2\cdot4^{r} -4^{s+1}\equiv 1\pmod p$.
It follows that
\begin{equation*}
(2\cdot\frac{1- (-2)^{2r}}{3}+ 2\cdot\frac{1- (-2)^{2r-1}}{3})n= n.
\end{equation*}
Now we either obtain a contradiction by \textbf{Case 1} or \textbf{Case 1'}, or by inequalities (\ref{E2k+1n}) and (\ref{E2kn}) we have
\begin{equation*}
K_1<f_1(t_{2r-1}(1)n)+f_2(t_{2r-1}(1)n),\ t_{2r-1}(1)n\neq 0,    
\end{equation*}
\begin{equation*}
  K_2<f_3(t_{2r-2}(1)n)+f_4(t_{2r-2}(1)n),\  t_{2r-2}(1)n\neq 0  
\end{equation*}
and
\begin{equation*}
    (t_{2r-1}(1)+t_{2r-1}(1)+t_{2r-2}(1)+t_{2r-2}(1))n= n,
\end{equation*}
which is also a contradiction.

Next, consider $q=p^k$ where $k>1$. Define $g_1,g_2,g_3,g_4$ over $\Z_p^*$ by
\begin{equation*}
  g_i(x)=p^{1-k}\sum_{a\equiv x\pmod p}f_i(a).   
\end{equation*}
For any $n\in\Z_{p^k}$, by the case $q=p$, there exist $x_1,x_2,x_3,x_4\in\Z_p^*$ such that $ n\equiv x_1+x_2+x_3+x_4 \pmod p$ and 
\begin{equation*}
    g_1(x_1)+g_2(x_2)\geq\E_{x\in\Z_p^*}(g_1(x)+g_2(x))=K_1,
\end{equation*}
\begin{equation*}
 g_3(x_3)+g_4(x_4)\geq\E_{x\in\Z_p^*}(g_3(x)+g_4(x))=K_2 .   
\end{equation*}

Let $n'=(n-x_1-x_2-x_3-x_4)/p\in\Z_{p^{k-1}}$, define $h_1,h_2,h_3,h_4$ over $\Z_{p^{k-1}}$ by $h_i(y)=f_i(x_i+yp)$, $ i=1,2,3,4$. Notice that for all $z_1\in \Z_{p^{k-1}}$ we have 
\begin{align*}
 \E_{y\in\Z_{p^{k-1}}}(h_1(y)+h_2(z_1-y))
= &\E_{y\in\Z_{p^{k-1}}}(h_1(y)+h_2(y))\\
= &\E_{y\in\Z_{p^{k-1}}}(f_1(x_1+yp)+f_2(x_2+yp))\\
= &p^{1-k}(\sum_{a\equiv x_1\pmod p}f_1(a)+\sum_{a\equiv x_2\pmod p}f_2(a))\\
=&g_1(x_1)+g_2(x_2) 
\geq K_1.
\end{align*}
Similarly we have $\E_{y\in\Z_{p^{k-1}}}(h_3(y)+h_4(z_2-y))=g_3(x_3)+g_4(x_4)\geq K_2$ for all $z_2\in \Z_{p^{k-1}}$.

Let $z_1+z_2\equiv n'\pmod {p^{k-1}}$, then there exist $y_1,y_2,y_3,y_4\in \Z_{p^{k-1}} $ such that $n'\equiv y_1+y_2+y_3+y_4\pmod {p^{k-1}}$ and 
\begin{align*}
h_1(y_1)+h_2(y_2)\geq\E_{y\in\Z_{p^{k-1}}}(h_1(y)+h_2(y))=K_1,\\
h_3(y_3)+h_4(y_4)\geq\E_{y\in\Z_{p^{k-1}}}(h_3(y)+h_4(y))=K_2,
\end{align*}
i.e.,
\begin{align*}
f_1(x_1+y_1p)+f_2(x_2+y_2p)\geq K_1,\\
f_3(x_3+y_3p)+f_4(x_4+y_4p)\geq K_2.
\end{align*}

Finally, consider $q=q_1q_2$ where $q_ 1$ and $q _2$ are coprime. Suppose that Theorem \ref{Theorem-2} holds for two coprime integers $q_1$ and $q_2$. We now show that it then holds for their product $q = q_1 q_2$ as well.

Consider $\Z_q$ as $\Z_{q_1}\oplus\Z_{q_2}$ and function $g_1,g_2,g_3,g_4$ over $\Z_{q_1}^*$by $g_i(x)=\frac{1}{\varphi(q_2)}\sum_{y\in\Z_{q_2}^*}f_i((x,y))$.
Then by the induction hypothesis, for any $n = (n_1, n_ 2 ) \in \Z_{q_1}\oplus\Z_{q_2}$, there exist $x_1,x_2,x_3,x_4\in\Z_{q_1}^*$ such that $n_1=x_1+x_2+x_3+x_4$ and 
\begin{align*}
g_1(x_1)+g_2(x_2)\geq\E_{x\in\Z_{q_1}^*}(g_1(x)+g_2(x))=K_1,\\ g_3(x_3)+g_4(x_4)\geq\E_{x\in\Z_{q_1}^*}(g_3(x)+g_4(x))=K_2,
\end{align*}
i.e.,
\begin{align*}
\frac{1}{\varphi(q_2)}\sum_{y\in\Z_{q_2}^*}f_1((x_1,y))+\frac{1}{\varphi(q_2)}\sum_{y\in\Z_{q_2}^*}f_2((x_2,y))\geq K_1,\\
\frac{1}{\varphi(q_2)}\sum_{y\in\Z_{q_2}^*}f_3((x_3,y))+\frac{1}{\varphi(q_2)}\sum_{y\in\Z_{q_2}^*}f_4((x_4,y))\geq K_2.
\end{align*}

Deﬁne functions $ h_1,h_2,h_3,h_4 $ over $\Z_{q_2}^*$ by $h_i(y)=f_i((x_i,y))$. Therefore 
\begin{align*}
\frac{1}{\varphi(q_2)}\sum_{y\in\Z_{q_2}^*}f_i((x_i,y))= \E_{y\in\Z_{q_2}^*}h_i(y).
\end{align*}

Applying the induction hypothesis again, there exist $y_1,y_2,y_3,y_4\in\Z_{q_2}^*$ such that $n_2=y_1+y_2+y_3+y_4$ and 
\begin{align*}
h_1(y_1)+h_2(y_2)\geq\E_{y\in\Z_{q_2}^*}(h_1(y)+h_2(y)) ,\\
h_3(y_3)+h_4(y_4)\geq\E_{y\in\Z_{q_2}^*}(h_3(y)+h_4(y)).
\end{align*}

Therefore,
\begin{gather*}
(n_1,n_2)=(x_1,y_1)+(x_2,y_2)+(x_3,y_3)+(x_4,y_4),\\
f_1((x_1,y_1))+f_2((x_2,y_2))\geq K_1,\\ 
f_3((x_3,y_3))+f_4((x_4,y_4))\geq K_2.
\end{gather*}

This finishes the proof by induction.
\end{proof} 

\begin{corollary}\label{coro1}
Let $q$ be an odd squarefree integer. Let $f_i:\Z_q^*\rightarrow \left[ 0,1 \right]$, $i=1,2,3,4$. Suppose that $\E_{x\in\Z_q^*}(f_1(x)+f_2(x))> 1$ and $\E_{x\in\Z_q^*}(f_3(x)+f_4(x))> 1$.    
Then for any integer $n$, there exists $x _1 ,\cdots, x_4\in\Z_q^*$ such that $n\equiv\sum_{i=1}^4 x_i \pmod q$, $\sum_{i=1}^4 f_i(x_i)> 2$ and $f_i(x_i)\neq 0$ for $i=1,2,3,4$.
\end{corollary}

\begin{proof}
Our proof is based on a case-by-case analysis. If $3 \nmid q$, then Corollary \ref{coro1} is a simple consequence of Theorem \ref{Theorem-2}. Therefore, we now consider $3 \mid q$.

\textbf{Case 1.} Assume $q = 3$, $n=0$.
If $f_i(j) \neq 0$ for all $i\in\{1,2,3,4\}$, $j\in\Z_3^*$, we have 
\begin{align*}
& \ \  \ \ \max \left\{ \sum_{i=1}^4 f_i(x_i) : x_1,\dots,x
_4\in\Z_3^* \ \text{and}  \ \sum_{i=1}^4 x_i \equiv 0 \pmod{3}\right\}\\
&= \max\{f_1(1)+f_2(1)+f_3(2)+f_4(2),f_1(1)+f_2(2)+f_3(1)+f_4(2),f_1(2)+f_2(1)+f_3(1)+f_4(2),\\
& \ \ \ \ \ \ \ f_1(1)+f_2(2)+f_3(2)+f_4(1),f_1(2)+f_2(1)+f_3(2)+f_4(1),f_1(2)+f_2(2)+f_3(1)+f_4(1)\}\\
&\geq \frac{1}{6}(3f_1(1)+3f_1(2)+3f_2(1)+3f_2(2)+3f_3(1)+3f_3(2)+3f_4(1)+3f_4(2))\\
&=\frac{f_1(1)+f_1(2)}{2}+ \frac{f_2(1)+f_2(2)}{2}+ \frac{f_3(1)+f_3(2)}{2}+ \frac{f_4(1)+f_4(2)}{2}\\
&>2.
\end{align*}

If there exists a $f_i(j) = 0$, without loss of generality we let $i=1,j=1$, then $f_1(2)+f_2(1)+f_2(2)>2$ and $ f_1(2)f_2(1)f_2(2)\neq0$. Then we have
\begin{align*}
&\max\{f_1(2)+f_2(1)+f_3(1)+f_4(2),f_1(2)+f_2(1)+f_3(2)+f_4(1)\}\\
&>1+ \max\{f_3(1)+f_4(2),f_3(2)+f_4(1)\}\\
&>2,        
\end{align*}
where $\max\{f_3(1)+f_4(2),f_3(2)+f_4(1)\}>1$ also implies that the larger of the two sums has both summands non-zero.

\textbf{Case 2.} Assume $q = 3$, $n=1$.
Suppose $f_i(j) \neq 0$ for all $i\in\{1,2,3,4\}$, $j\in\Z_3^*$. If $f_1(2)+f_2(2)+f_3(2)+f_4(2)\leq2$ then $f_1(1)+f_2(1)+f_3(1)+f_4(1)>2$; the latter inequality is exactly what we need. Hence we assume $f_1(2)+f_2(2)+f_3(2)+f_4(2)>2$. Now we have 
\begin{align*}
& \ \  \ \ \max \left\{ \sum_{i=1}^4 f_i(x_i) : x_1,\dots,x
_4\in\Z_3^* \ \text{and}  \ \sum_{i=1}^4 x_i \equiv 1 \pmod{3}\right\}\\
&= \max\{f_1(1)+f_2(1)+f_3(1)+f_4(1),f_1(1)+f_2(2)+f_3(2)+f_4(2),f_1(2)+f_2(1)+f_3(2)+f_4(2),\\
& \ \ \ \ \ \ \ f_1(2)+f_2(2)+f_3(1)+f_4(2),f_1(2)+f_2(2)+f_3(2)+f_4(1)\}\\
&\geq \frac{1}{5}(2f_1(1)+3f_1(2)+2f_2(1)+3f_2(2)+2f_3(1)+3f_3(2)+2f_4(1)+3f_4(2))\\
&>\frac{8}{5}+\frac{1}{5}(f_1(2)+f_2(2)+f_3(2)+f_4(2))\\
&>2.        
\end{align*}

If there exists a $f_i(1) = 0$, without loss of generality we let $i=1$, then $f_1(2)+f_2(1)+f_2(2)>2$ and $ f_1(2)f_2(1)f_2(2)\neq0$. Then we have
\begin{align*}
&\max\{f_1(2)+f_2(2)+f_3(1)+f_4(2),f_1(2)+f_2(2)+f_3(2)+f_4(1)\}\\
&>1+ \max\{f_3(1)+f_4(2),f_3(2)+f_4(1)\}\\
&>2,        
\end{align*}
where $\max\{f_3(1)+f_4(2),f_3(2)+f_4(1)\}>1$ also implies that the larger of the two sums has both summands non-zero.

If there exists a $f_i(2) = 0$, without loss of generality we let $i=1$, then $f_1(1)+f_2(1)+f_2(2)>2$ and $ f_1(1)f_2(1)f_2(2)\neq0$. Then we have
\begin{align*}
& \max\{f_1(1)+f_2(1)+f_3(1)+f_4(1),f_1(1)+f_2(2)+f_3(2)+f_4(2)\}\\
&>1+ \max\{f_3(1)+f_4(1),f_3(2)+f_4(2)\}\\
&>2,        
\end{align*}
where $\max\{f_3(1)+f_4(1),f_3(2)+f_4(2)\}>1$ also implies that the larger of the two sums has both summands non-zero.

\textbf{Case 3.} Assume $q = 3$, $n=2$.
Suppose $f_i(j) \neq 0$ for all $i\in\{1,2,3,4\}$, $j\in\Z_3^*$. If $f_1(1)+f_2(1)+f_3(1)+f_4(1)\leq2$ then $f_1(2)+f_2(2)+f_3(2)+f_4(2)>2$; the latter inequality is exactly what we need. Hence we assume $f_1(1)+f_2(1)+f_3(1)+f_4(1)>2$. Now we have 
\begin{align*}
& \ \  \ \ \max \left\{ \sum_{i=1}^4 f_i(x_i) : x_1,\dots,x
_4\in\Z_3^* \ \text{and}  \ \sum_{i=1}^4 x_i \equiv 2 \pmod{3}\right\}\\
&= \max\{f_1(1)+f_2(1)+f_3(1)+f_4(2),f_1(1)+f_2(1)+f_3(2)+f_4(1),f_1(1)+f_2(2)+f_3(1)+f_4(1),\\
& \ \ \ \ \ \ \ f_1(2)+f_2(1)+f_3(1)+f_4(1),f_1(2)+f_2(2)+f_3(2)+f_4(2)\}\\
&\geq \frac{1}{5}(3f_1(1)+2f_1(2)+3f_2(1)+2f_2(2)+3f_3(1)+2f_3(2)+3f_4(1)+2f_4(2))\\
&>\frac{8}{5}+\frac{1}{5}(f_1(1)+f_2(1)+f_3(1)+f_4(1))\\
&>2.        
\end{align*}

If there exist a $f_i(1) = 0$, without loss of generality we let $i=1$, then $f_1(2)+f_2(1)+f_2(2)>2$ and $ f_1(2)f_2(1)f_2(2)\neq0$. Then we have
\begin{align*}
&\max\{f_1(2)+f_2(1)+f_3(1)+f_4(1),f_1(2)+f_2(2)+f_3(2)+f_4(2)\}\\
&>1+ \max\{f_3(1)+f_4(1),f_3(2)+f_4(2)\}\\
&>2,        
\end{align*}
where $\max\{f_3(1)+f_4(1),f_3(2)+f_4(2)\}>1$ also implies that the larger of the two sums has both summands non-zero.

If there exist a $f_i(2) = 0$, without loss of generality we let $i=1$, then $f_1(1)+f_2(1)+f_2(2)>2$ and $ f_1(1)f_2(1)f_2(2)\neq0$. Then we have
\begin{align*}
& \max\{f_1(1)+f_2(1)+f_3(1)+f_4(2),f_1(1)+f_2(1)+f_3(2)+f_4(1)\}\\
&>1+ \max\{f_3(1)+f_4(2),f_3(2)+f_4(1)\}\\
&>2,       
\end{align*}
where $\max\{f_3(1)+f_4(2),f_3(2)+f_4(1)\}>1$ also implies that the larger of the two sums has both summands non-zero.

\textbf{Case 4.} Assume $q = 3q'$, where $3\nmid q'$. By Theorem \ref{Theorem-2}, for any $n=(n_1,n_2)\in \Z_{q'} \oplus\Z_3$ there exist $x_1, x_2, x_3,x_4 \in\Z _{q'}^*$ such that $n_1 = x_1+ x_2+ x_3+x_4 $ and
\begin{align*}
  \E_{y\in\Z_3^*}(f_1((x_1,y))+f_2((x_2,y)))>1, \\
  \E_{y\in\Z_3^*}(f_3((x_3,y))+f_4((x_4,y)))>1  .
\end{align*}

Combining \textbf{Case 1-3}, there exist $y_1,y_2,y_3,y_4 \in\Z _3 ^*$ such that $n_ 2 =y_1+y_2+y_3+y_4 $,
$$
f_1((x_1,y_1))+f_2((x_2,y_2))+f_3((x_3,y_3))+f_4((x_4,y_4))>2,
$$
    and
$$
f_1((x_1,y_1))f_2((x_2,y_2))f_3((x_3,y_3))f_4((x_4,y_4))\neq 0.
$$

\end{proof}

In particular, let $f_i$ be the characteristic functions of $A_i$, we have the following generalization of Corollary 1.3 in \cite{Gao1}.

\begin{corollary}
Let $q$ be an odd squarefree integer. Let $A_1, A_2, A_3, A_4$ be four subsets of $\Z_q^*$ with $|A_1|+|A_2|>\varphi (q)$ and $|A_3|+|A_4|>\varphi (q)$. Then $A_1+A_2+A_3+A_4=\Z _q$.
\end{corollary}

\section{Proof of Theorem \ref{Theorem-1}}
Let 
\begin{equation*}
	\kappa=\min\{10^{-4}(\underline{\delta}(P_1)+\underline{\delta}(P_2)-1),10^{-4}(\underline{\delta}(P_3)+\underline{\delta}(P_4)-1)\}
\end{equation*}
and
\begin{equation*}
	\alpha_i=\underline{\delta}(P_i)/(1+2\kappa)
\end{equation*}
for $i\in\{1,2,3,4\}$. Let $n$ be a sufficiently large positive integer. By the definition of $\underline{\delta}(P_i)$, we have
\begin{equation}\label{lower bound for P_i (1)}
	|P_{i}\cap [1,n/2]|\geq (1+\kappa)\alpha_i\dfrac{n/2}{\log n}.
\end{equation}
Let $w(n)=\frac{1}{4}\log \log n$ and
\begin{equation*}
	W=\prod_{p \leq w(n)}p.
\end{equation*}
Clearly $W\leq \log n$. By \eqref{lower bound for P_i (1)}, we have
\begin{equation}\label{lower bound for P_i (2)}
	\begin{aligned}
		\sum_{\substack{x\leq n/2\\ (x,W)=1}}\mathbf{1}_{P_{i}}(x)\log x &\geq \sum_{n^{\frac{1}{1+\kappa/2}}\leq x\leq n/2}\mathbf{1}_{P_{i}}(x)\log x\\
		&\geq \dfrac{\log n}{1+\kappa/2}\bigg( \dfrac{(1+\kappa)\alpha_i(n/2)}{\log n}-n^{\frac{1}{1+\kappa/2}} \bigg)\\
		&\geq \dfrac{1}{2}\alpha_i n.
	\end{aligned}
\end{equation}
For $b\in\mathbb{Z}_{W}^{\ast}$, define
\begin{equation*}
	f_{i}(b)=\max \bigg\{ 0,\dfrac{2\varphi(W)}{n}\sum_{\substack{x\leq n/2\\ x\equiv b \pmod{W}}}\mathbf{1}_{P_{i}}(x)\log x-3\kappa \bigg\}.
\end{equation*}
It follows from Siegel--Walfisz theorem that $f_{i}(b)\in [0,1]$ for all $b\in\mathbb{Z}_{W}^{\ast}$. Also, by \eqref{lower bound for P_i (2)}, we have
\begin{align*}
	\sum_{b\in\mathbb{Z}_{W}^{\ast}}(f_1(b)+f_2(b))&\geq \dfrac{2\varphi(W)}{n}\sum_{b\in\mathbb{Z}_{W}^{\ast}}\sum_{\substack{x\leq n/2\\ x\equiv b \pmod{W}}}(\mathbf{1}_{P_{1}}(x)+\mathbf{1}_{P_{2}}(x))\log x-6\kappa\varphi(W)\\
	&\geq (\alpha_1+\alpha_2-6\kappa)\varphi(W)>\varphi(W).
\end{align*}
Similarly, we have
\begin{equation*}
	\sum_{b\in\mathbb{Z}_{W}^{\ast}}(f_3(b)+f_4(b))>\varphi(W).
\end{equation*}
It follows from Corollary \ref{coro1} that there exist $b_1, b_2, b_3, b_4\in \mathbb{Z}_{W}^{\ast}$ such that $n\equiv b_1+b_2+b_3+b_4\pmod{W}$, $f_1(b_1)+f_2(b_2)+f_3(b_3)+f_4(b_4)>2$ and $f_i(b_i)>0$ for $1\leq i\leq 4$.  And without loss of generality, we may assume that $1\leq b_1, b_2, b_3, b_4<W$.

Let $N$ be a prime and $(1+\kappa)n/W\leq N\leq (1+2\kappa)n/W$. By the Prime Number Theorem, we know that such $N$ always exists for sufficiently large $n$. Let $n'=(n-b_1-b_2-b_3-b_4)/W$ and
\begin{equation*}
	A_{i}=\{ x:Wx+b_i\in P_i\cap [1,n/2] \}.
\end{equation*}
It suffices to show that $n'\in A_1+A_2+A_3+A_4$. Define
\begin{equation*}
	\lambda_{b, W, N}(x) =
	\begin{cases} 
		\varphi(W) \log(Wx + b)/WN & \text{if } x \leq N \text{ and } Wx + b \text{ is prime}, \\
		0 & \text{otherwise}.
	\end{cases}
\end{equation*}
and
\begin{equation*}
	\alpha'_i=\sum_{x}\mathbf{1}_{A_i}(x)\lambda_{b_i, W, N}(x).
\end{equation*}
Then we have
\begin{equation}\label{lower bound (1)}
	\begin{aligned}
		\alpha'_i&=\dfrac{\varphi(W)}{WN}\sum_{x}\mathbf{1}_{A_i}(x)\log (Wx+b_i)=\dfrac{\varphi(W)}{WN}\sum_{\substack{x\leq n/2\\ x\equiv b_i \pmod{W}}}\mathbf{1}_{P_i}(x)\log x\\
		&=\dfrac{n/2}{WN}(f_i(b_i)+3\kappa)\geq \dfrac{f_i(b_i)+3\kappa}{2(1+2\kappa)}\geq \kappa 
	\end{aligned}
\end{equation}
for $1\leq i\leq 4$ and
\begin{equation}\label{lower bound (2)}
	\alpha'_1+\alpha'_2+\alpha'_3+\alpha'_4\geq \dfrac{1}{2(1+2\kappa)}(f_1(b_1)+f_2(b_2)+f_3(b_3)+f_4(b_4)+12\kappa)>\dfrac{1+6\kappa}{1+2\kappa}\geq 1+3\kappa.
\end{equation}
Below we consider $A_1,A_2,A_3,A_4$ as subsets of $\mathbb{Z}_N$. Since $A_1,A_2,A_3,A_4\subseteq [1,n/2W]$ and $N\geq (1+\kappa)n/W$, there is no nonzero integer $k$ such that $kN+n'=x_1+x_2+x_3+x_4$, where $x_i\in A_i$. Therefore $n'\in A_1+A_2+A_3+A_4$ in $\mathbb{Z}_{N}$ implies that $n'\in A_1+A_2+A_3+A_4$ in $\mathbb{Z}$. The rest of the analysis is done in $\mathbb{Z}_N$.

Let $\mu_i(x)=\lambda_{b_i, W, N}(x)$ and $a_i(x)=\mathbf{1}_{A_i}(x)\mu_i(x)$. Suppose that $\delta, \varepsilon > 0$ are two real numbers which will be chosen later. Let
\begin{equation*}
	R_i = \{ r \in \mathbb{Z}_N : |\tilde{a}_i(r)| \geq \delta \}
\end{equation*}
and
\begin{equation*}
	B_i = \{ x \in \mathbb{Z}_N : \| xr/N \| \leq \varepsilon \text{ for all } r \in R_i \},
\end{equation*}
where $\|x\| = \min_{z \in \mathbb{Z}} |x - z|$. Let $\beta_i = \mathbf{1}_{B_i} / |B_i|$ and $a_i' = a_i \ast \beta_i \ast \beta_i$.

Next, we prove that the functions $a_1\ast a_2\ast a_3\ast a_4$ and $a_1'\ast a_2'\ast a_3'\ast a_4'$ are in a certain sense close to each other.
\begin{lemma}\label{Approximation of convolution}
	We have the estimate
	\begin{equation*}
		\left| \sum_{\substack{x_1,x_2,x_3,x_4 \in \mathbb{Z}_N\\ x_1+x_2+x_3+x_4=n'}} \prod_{i=1}^{4}a_i'(x_i) - \sum_{\substack{x_1,x_2,x_3,x_4 \in \mathbb{Z}_N\\ x_1+x_2+x_3+x_4=n'}} \prod_{i=1}^{4}a_i(x_i) \right| \leq \frac{C_1}{N} (\epsilon^2 \delta^{-5/2} + \delta^{1/5}).
	\end{equation*}
\end{lemma}
\begin{proof}
	Note that
	\begin{align*}
		\sum_{\substack{x_1,x_2,x_3,x_4 \in \mathbb{Z}_N\\ x_1+x_2+x_3+x_4=n'}} \prod_{i=1}^{4}a_i(x_i)&=\dfrac{1}{N}\sum_{x_1,x_2,x_3,x_4 \in \mathbb{Z}_N}\prod_{i=1}^{4}a_i(x_i)\sum_{r\in\mathbb{Z}_N}e((n'-x_1-x_2-x_3-x_4)r/N)\\
		&=\dfrac{1}{N}\sum_{r\in\mathbb{Z}_N}\prod_{i=1}^{4}\tilde{a_i}(r)e(n'r/N)
	\end{align*}
	and
	\begin{align*}
		\sum_{\substack{x_1,x_2,x_3,x_4 \in \mathbb{Z}_N\\ x_1+x_2+x_3+x_4=n'}} \prod_{i=1}^{4}a_i'(x_i)&=\dfrac{1}{N}\sum_{x_1,x_2,x_3,x_4 \in \mathbb{Z}_N}\prod_{i=1}^{4}a_i'(x_i)\sum_{r\in\mathbb{Z}_N}e((n'-x_1-x_2-x_3-x_4)r/N)\\
		&=\dfrac{1}{N}\sum_{r\in\mathbb{Z}_N}\prod_{i=1}^{4}\tilde{a_i'}(r)e(n'r/N)=\dfrac{1}{N}\sum_{r\in\mathbb{Z}_N}\prod_{i=1}^{4}\tilde{a_i}(r)\tilde{\beta_i}(r)^{2}e(n'r/N).
	\end{align*}
	We have
	\begin{align*}
		&\left| \sum_{\substack{x_1,x_2,x_3,x_4 \in \mathbb{Z}_N\\ x_1+x_2+x_3+x_4=n'}} \prod_{i=1}^{4}a_i'(x_i) - \sum_{\substack{x_1,x_2,x_3,x_4 \in \mathbb{Z}_N\\ x_1+x_2+x_3+x_4=n'}} \prod_{i=1}^{4}a_i(x_i) \right|\\&=\dfrac{1}{N}\left|\sum_{r\in \mathbb{Z}_{N}}\bigg( \prod_{i=1}^{4}\tilde{a_i}(r) \bigg)\bigg( 1-\prod_{i=1}^{4}\tilde{\beta_i}^{2}(r) \bigg)e(n'r/N)\right|.
	\end{align*}
	By \cite[proof of Lemma 6.7]{Gre}, we know that
	\begin{equation*}
		|1-\tilde{\beta_i}(r)|\leq 16\varepsilon^{2}
	\end{equation*}
	for $r\in R_i$. Therefore, if $r\in R=R_1\cap R_2\cap R_3\cap R_4$, then
	\begin{equation*}
		\left|1-\prod_{i=1}^{4}\tilde{\beta_i}^{2}(r) \right|\leq C_2\varepsilon^2
	\end{equation*}
	for some constant $C_2>0$. Also, by \cite[proof of Proposition 6.4]{Gre}, we know that $|R_i|\leq C_3\delta^{-5/2}$ for some constant $C_3>0$.
	Note that $\left|\tilde{a_i}(r)\right|\leq \sum_{x\in\mathbb{Z}_N}a_i(x)\leq 1$. We have
	\begin{align*}
		\left|\sum_{r\in R}\bigg( \prod_{i=1}^{4}\tilde{a_i}(r) \bigg)\bigg( 1-\prod_{i=1}^{4}\tilde{\beta_i}^{2}(r) \bigg)e(n'r/N)\right|&\leq C_2\varepsilon^2\sum_{r\in R}\left|\tilde{a_1}(r)\tilde{a_2}(r)\tilde{a_3}(r)\tilde{a_4}(r)\right|\\
		&\leq C_2\varepsilon^2|R|\leq C_2C_3\varepsilon^2\delta^{-5/2}.
	\end{align*}
	By the H\"older's inequality and noting that $|\tilde{\beta_i}(r)|\leq \sum_{x\in \mathbb{Z}_{N}}\beta_i(x)=1$, we have
	\begin{align*}
		\left|\sum_{r\notin R}\bigg( \prod_{i=1}^{4}\tilde{a_i}(r) \bigg)\bigg( 1-\prod_{i=1}^{4}\tilde{\beta_i}^{2}(r) \bigg)e(n'r/N)\right|&\leq 2\sup_{r\notin R}\left|\prod_{i=1}^{4}\tilde{a_i}(r)\right|^{1/5}\sum_{r\notin R}\left|\prod_{i=1}^{4}\tilde{a_i}(r)\right|^{4/5}\\
		&\leq 2\delta^{1/5}\prod_{i=1}^{4}\bigg( \sum_{r\notin R}\left|\tilde{a_i}(r)\right|^{16/5} \bigg)^{1/4}.
	\end{align*}
	Applying \cite[Lemma 6.6]{Gre} with $p=16/5$, we have
	\begin{equation*}
		\left|\sum_{r\notin R}\bigg( \prod_{i=1}^{4}\tilde{a_i}(r) \bigg)\bigg( 1-\prod_{i=1}^{4}\tilde{\beta_i}^{2}(r) \bigg)e(n'r/N)\right|\leq C_4\delta^{1/5}.
	\end{equation*}
	This completes the proof.
\end{proof}

\begin{lemma}\label{upper bound for convolution}
	Suppose that $\varepsilon^{|R_i|} \geq C_{5} \log \log w / w$. Then for each $x \in \mathbb{Z}_N$,
	\begin{equation*}
		|a_i'(x)| \leq (1 + 2/C_{5})/N.
	\end{equation*}
\end{lemma}
\begin{proof}
As the proof is the same as that of \cite[Lemma 6.3]{Gre}, we omit it here.
\end{proof}
\begin{lemma}\label{sumset}
	Suppose that $k\geq 2$ and $0 < \theta_1, \ldots, \theta_k \leq 1$ with $\theta_1 + \cdots + \theta_k > 1$. Let
	\begin{equation*}
		\theta = \min\{\theta_1, \ldots, \theta_k, (\theta_1 + \cdots + \theta_k - 1)/(3k - 5)\}.
	\end{equation*}
	Suppose that $N$ is a prime greater than $2\theta^{-2}$, and $X_1, \ldots, X_k$ are subsets of $\mathbb{Z}_N$ with $|X_i| \geq \theta_i N$. Then for any $n \in \mathbb{Z}_N$, we have
	\begin{equation*}
		|\{(x_1, x_2, \ldots, x_k) : x_i \in X_i, \; n = x_1 + x_2 + \dots + x_k\}|\geq \theta^{2k-3} N^{k-1}.
	\end{equation*}
\end{lemma}
\begin{proof}
	See \cite[Lemma 3.3]{LP}.
\end{proof}
Next, we give a lower bound for $a_1'\ast a_2'\ast a_3'\ast a_4'(n')$. Let $A'_i = \{x \in \mathbb{Z}_N : a'_i(x) \geq \alpha'_i\kappa/N\}$. Applying Lemma \ref{upper bound for convolution} with $C_{5} = 2/\kappa$, we have
\begin{equation*}
	\alpha'_i = \sum_{x \in \mathbb{Z}_N} a_i(x) = \sum_{x \in \mathbb{Z}_N} a'_i(x) \le \frac{1+\kappa}{N} |A'_i| + \frac{\alpha'_i\kappa}{N}(N - |A'_i|),
\end{equation*}
which implies that
\begin{equation*}
	|A'_i| \geq \frac{\alpha'_i(1-\kappa)}{1+\kappa} N.
\end{equation*}
By \eqref{lower bound (1)} and \eqref{lower bound (2)}, we have
\begin{equation*}
	\frac{\alpha'_i(1-\kappa)}{1+\kappa}\geq \dfrac{\kappa}{2}
\end{equation*}
and
\begin{equation*}
	\sum_{i=1}^{4}\dfrac{\alpha'_i(1-\kappa)}{1+\kappa}=\dfrac{1-\kappa}{1+\kappa}(\alpha'_1+\alpha'_2+\alpha'_3+\alpha'_4)\geq 1+\kappa/2.
\end{equation*}
Therefore, by Lemma \ref{sumset}, we have
\begin{equation*}
	|\{(x_1, x_2, x_3, x_4) : x_i \in A'_i, \; n' = x_1 + x_2 + x_3 + x_4\}|\geq C_{6}(\kappa) N^{3},
\end{equation*}
where $C_{6}(\kappa)>0$ is a constant depending only on $\kappa$. It follows that
\begin{equation*}
	\sum_{\substack{x_1,x_2,x_3,x_4 \in \mathbb{Z}_N\\ x_1+x_2+x_3+x_4=n'}} \prod_{i=1}^{4}a_i'(x_i)\geq \sum_{\substack{x_i\in A'_i\\ x_1+x_2+x_3+x_4=n'}} \prod_{i=1}^{4}a_i'(x_i)\geq \alpha'_1\alpha'_2\alpha'_3\alpha'_4\kappa^{4}C_{6}(\kappa)N^{-1}.
\end{equation*}
By \eqref{lower bound (1)}, we have
\begin{equation*}
	\sum_{\substack{x_1,x_2,x_3,x_4 \in \mathbb{Z}_N\\ x_1+x_2+x_3+x_4=n'}} \prod_{i=1}^{4}a_i'(x_i)\geq \dfrac{C_{7}(\kappa)}{N},
\end{equation*}
where $C_{7}(\kappa)>0$ is a constant depending only on $\kappa$. By Lemma \ref{Approximation of convolution}, we have
\begin{equation*}
	\sum_{\substack{x_1,x_2,x_3,x_4 \in \mathbb{Z}_N\\ x_1+x_2+x_3+x_4=n'}} \prod_{i=1}^{4}a_i(x_i)\geq \dfrac{C_7(\kappa)-C_1\cdot(\varepsilon^2 \delta^{-5/2} + \delta^{1/5})}{N}.
\end{equation*}
We choose $\delta$ and $\varepsilon$ such that both $\varepsilon^2 \delta^{-5/2}$ and $\delta^{1/5}$ tend to $0$, whenever $N$ is sufficiently large. Thus for sufficiently large $n$,
\begin{equation*}
	\sum_{\substack{x_1,x_2,x_3,x_4 \in \mathbb{Z}_N\\ x_1+x_2+x_3+x_4=n'}} \prod_{i=1}^{4}a_i(x_i)>0.
\end{equation*}
This completes the proof of Theorem \ref{Theorem-1}.

\end{document}